\documentclass[english]{article}
\usepackage{babel}
\usepackage{xcolor}
\usepackage{units}
\usepackage{amsmath}
\usepackage{amsthm}
\usepackage{amssymb}

\makeatletter

\theoremstyle{plain}
\newtheorem{thm}{\protect\theoremname}
  \theoremstyle{plain}
  \newtheorem{lem}[thm]{\protect\lemmaname}
  \theoremstyle{plain}
  \newtheorem{prop}[thm]{\protect\propositionname}
  \theoremstyle{remark}
  \newtheorem{rem}[thm]{\protect\remarkname}

\makeatother

  \providecommand{\lemmaname}{Lemma}
  \providecommand{\propositionname}{Proposition}
  \providecommand{\remarkname}{Remark}
\providecommand{\theoremname}{Theorem}

\begin{document}

\title{Singularity of generalized grey Brownian motions with different parameters}

\author{\textbf{Jos\textbf{{\'e}} Lu\textbf{{\'\i}}s da Silva},\\
 CIMA, University of Madeira, Campus da Penteada,\\
 9020-105 Funchal, Portugal.\\
 Email: luis@uma.pt\and \textbf{Mohamed Erraoui}\\
 Universit{\'e} Cadi Ayyad, Facult{\'e} des Sciences Semlalia,\\
 D{\'e}partement de Math{\'e}matiques, BP 2390, Marrakech, Maroc\\
 Email: erraoui@uca.ma}
\maketitle
\begin{abstract}
In this note we prove that the probability measures generated by two
generalized grey Brownian motions with different parameters are singular
with respect to each other. This result can be interpreted as an extension
of the Feldman-H{\'a}jek dichotomy of Gaussian measures to a family
of non-Gaussian measures.\medskip{}

\end{abstract}

\section{Introduction }

According to the Feldman-H{\'a}jek dichotomy (cf.\ \cite{Feldman58,Hajek58})
a pair of Gaussian measures on the space of functions on an interval
are either mutually singular or equivalent. This result has led to
numerous attempts to obtain convenient criteria for deciding between
the two possibilities. In a paper on estimation of the Hurst index
for fractional Brownian motion (fBm for short) \cite{Kurchenko03},
Kurchenko derived a Baxter-type theorem for the fBm based on the second
order increments of the process. Based on Kurchenko's result, Prakasa
Rao, in \cite{Rao2008}, has investigated sufficient conditions under
which probability measures generated by two fBms are singular with
respect to each other. Inspired by the work of Cameron and Martin
\cite{Cameron-Martin47}, we provide another proof of Prakasa Rao's
result for fBm. To this end, we use the results on the variation of
fBm from Rogers, cf.\ \cite{Rogers1997}. This is the content of
Section\ \ref{sec:singularity-fBm}.

In the second part of this work we establish a similar result for
a $(\alpha,\beta)$-family of non-Gaussian measures, so-called generalized
grey noise measures (ggnm) $\mu_{\alpha,\beta}$ associated to the
generalized grey Brownian motion (ggBm) $B_{\alpha,\beta}$. To accomplished
that, at first we study the $p$-variation of the ggBm and then prove
the singularity of two ggnm with different parameters, cf.\ Section\ \ref{sec:ggBm}.

\section{Singularity of fractional Brownian motions with different Hurst indices}

\label{sec:singularity-fBm}

Let $\left\{ W_{H}(t),t\in[0,1]\right\} $ be standard fBm with Hurst
index $H\in(0,1)$ defined on a complete probability space $\left(\Omega,\mathcal{F},P\right)$.
Let $C([0,1]):=\big(C([0,1]),\mathcal{B}\big)$ be the space of continuous
functions with the $\sigma$-algebra $\mathcal{B}$ generated by the
cylinder sets. We denote by $P_{W_{H}}$ the probability measure induced
by $W_{H}$ on $C([0,1])$. The self-similarity of the fBm and the
ergodic theorem imply that the fBm has $\nicefrac{1}{H}$-variation
on the time interval $[0,1]$, see Rogers \cite{Rogers1997}. That
is, with probability $1$, we have 
\begin{equation}
\lim_{n\rightarrow\infty}\sum_{j=1}^{2^{n}}\left|W_{H}\left(\frac{j}{2^{n}}\right)-W_{H}\left(\frac{j-1}{2^{n}}\right)\right|^{\nicefrac{1}{H}}=\mathbb{E}\big(\big|W_{H}(1)\big|^{\nicefrac{1}{H}}\big)=:\mu_{H}.\label{eq:variation}
\end{equation}
Equivalently, for $P_{W_{H}}$ almost all $x\in C\big([0,1]\big)$
, we have 
\begin{equation}
\lim_{n\rightarrow\infty}\sum_{j=1}^{2^{n}}\left|x\left(\dfrac{j}{2^{n}}\right)-x\left(\dfrac{j-1}{2^{n}}\right)\right|^{\nicefrac{1}{H}}=\mu_{H}.\label{eq:equiv-variation}
\end{equation}
For each $p\in]0,+\infty)$ we introduce the following notations: 
\begin{enumerate}
\item Let $D^{p}$ be the subset of $C\big([0,1]\big)$ given by
\[
D^{p}:=\left\{ x\in C\big([0,1]\big)\,\Big|\,\lim_{n\rightarrow\infty}\sum_{j=1}^{2^{n}}\left|x\left(\frac{j}{2^{n}}\right)-x\left(\frac{j-1}{2^{n}}\right)\right|^{p}\;\mathrm{does\;not\;exist}\right\} .
\]
\item For each $\lambda\geq0$ let $C_{\lambda}^{p}$ be the subset of $C\big([0,1]\big)$
defined by 
\[
C_{\lambda}^{p}:=\left\{ x\in C\big([0,1]\big)\,\Big|\,\lim_{n\rightarrow\infty}\sum_{j=1}^{2^{n}}\left|x\left(\frac{j}{2^{n}}\right)-x\left(\frac{j-1}{2^{n}}\right)\right|^{p}=\lambda\right\} .
\]
\end{enumerate}
Now we give some properties satisfied by the sets $C_{\lambda}^{p},\,\lambda\geq0$
and $D^{p}$.
\begin{enumerate}
\item[\textbf{(i)}] The sets $C_{\lambda}^{p},\,\lambda\geq0$ and $D^{p}$ are all disjoint,
and 
\begin{equation}
C\big([0,1]\big)=\left(\bigcup_{0\leq\lambda}C_{\lambda}^{p}\right)\cup D^{p}.\label{eq:decomposition}
\end{equation}
\item[\textbf{(ii)}] For all $q>p$ and $\lambda\geq0$, we have 
\begin{equation}
C_{\lambda}^{p}\subset C_{0}^{q}.\label{incl prop}
\end{equation}

Indeed, for any $x\in C_{\lambda}^{p}$ yields 
\begin{align}
\sum_{j=1}^{2^{n}}\left|x\left(\dfrac{j}{2^{n}}\right)-x\left(\dfrac{j-1}{2^{n}}\right)\right|^{q} & \leq\left(\sup_{1\leq j\leq2^{n}}\left|x\left(\dfrac{j}{2^{n}}\right)-x\left(\dfrac{j-1}{2^{n}}\right)\right|^{q-p}\right)\label{eq:variation inequality set}\\
 & \times\sum_{j=1}^{2^{n}}\left|x\left(\dfrac{j}{2^{n}}\right)-x\left(\dfrac{j-1}{2^{n}}\right)\right|^{p}.\nonumber 
\end{align}
Now using the uniform continuity of $x$ in $[0,1]$, the right hand
side of \eqref{eq:variation inequality set} converge to $0$ as $n\rightarrow\infty$.
Hence the inclusion (\ref{incl prop}) follows.
\item[\textbf{(iii)}] By passing at complements in (\ref{incl prop}), we equivalently
have 
\[
\left(\bigcup_{0<\lambda}C_{\lambda}^{q}\right)\cup D^{q}\subset D^{p},\quad\forall q>p.
\]
\end{enumerate}
An interpretation of ${\bf {(ii)}}$ and ${\bf {(iii)}}$ is: if the
$p$-variation exists and it is nonzero then for any $q>p$ the $q$-variation
is zero and for any $q<p$ the $q$-variation is infinite.

A a consequence of (\ref{eq:equiv-variation}), the $P_{W_{H}}$-measures
of the sets $C_{\lambda}^{\nicefrac{1}{H}},\,\lambda\geq0$ and $D^{\nicefrac{1}{H}}$
are as follows:
\begin{lem}
\label{set meas} The $P_{W_{H}}$-measures of the sets $C_{\lambda}^{\nicefrac{1}{H}},\,\lambda\geq0$
and $D^{\nicefrac{1}{H}}$ are: 
\[
P_{W_{H}}\left(\bigcup_{0\leq\lambda\neq\mu_{H}}C_{\lambda}^{\nicefrac{1}{H}}\right)=P_{W_{H}}(D^{\nicefrac{1}{H}})=0,\quad P_{W_{H}}(C_{\mu_{H}}^{\nicefrac{1}{H}})=1.
\]
\end{lem}
Let $\big\{ W_{H_{i}}(t),t\in[0,1]\big\}$, $i=1,2$ be two fBms with
Hurst indices $H_{1}\neq H_{2}$. We will now show that the probability
measures induced by these processes $P_{W_{H_{1}}}$ , $P_{W_{H_{2}}}$
are singular with respect to each other. We state it in Theorem\ \ref{thm:singular-fBm-meaures}
below. The proof is based on the following lemma. 
\begin{lem}
\label{lem: sing}Assume that $H_{1}<H_{2}$. Then we have
\begin{enumerate}
\item $C_{\mu_{H_{2}}}^{\nicefrac{1}{H_{2}}}\subset C_{0}^{\nicefrac{1}{H_{1}}}$, 
\item $C_{\mu_{H_{1}}}^{\nicefrac{1}{H_{1}}}\subset D^{\nicefrac{1}{H_{2}}}$. 
\end{enumerate}
\end{lem}
\begin{proof}
\label{proof-lemma}The assertions 1. and 2. are consequences of ${\bf {(ii)}}$
and ${\bf {(iii)}}$ respectively.
\end{proof}
Now we are able to state the main result of this section 
\begin{thm}
\label{thm:singular-fBm-meaures}The probability measures $P_{W_{H_{1}}}$
and $P_{W_{H_{2}}}$ are singular with respect to each other. 
\end{thm}
\begin{proof}
It follows from Lemmas \ref{set meas} and \ref{lem: sing} that
$P_{W_{H_{1}}}\big(C_{\mu_{H_{2}}}^{\nicefrac{1}{H_{2}}}\big)=P_{W_{H_{2}}}\big((C_{\mu_{H_{2}}}^{\nicefrac{1}{H_{2}}})^{c}\big)=0$,
which means that the probability measures $P_{W_{H_{1}}}$ and $P_{W_{H_{2}}}$
are singular. 
\end{proof}

\section{Generalized grey Brownian motion}

\label{sec:ggBm}

Let $0<\alpha<2$ and $0<\beta\leq1$ be given. A continuous stochastic
process is a generalized grey Brownian motion noted by $\{B_{\alpha,\beta},\,t\geq0\}$,
if :
\begin{enumerate}
\item $B_{\alpha,\beta}(0)=0$ $P$-a.s.
\item Any collection $X=\big\{ B_{\alpha,\beta}(t_{1}),\ldots,B_{\alpha,\beta}(t_{n})\big\}$
with $0\leq t_{1}<t_{2}<\ldots<t_{n}<\infty$ has characteristic function
given by 
\begin{equation}
\mathbb{E}\left(\exp\left(i\sum_{i=1}^{n}\theta_{i}B_{\alpha,\beta}(t_{i})\right)\right)=E_{\beta}\left(-\frac{1}{2}\theta^{\top}\Sigma_{\alpha}\theta\right),\quad\label{cract func}
\end{equation}
\[
\theta=(\theta_{1},\ldots,\theta_{n})\in\mathbb{R}^{n},\;\Sigma_{\alpha}=\big(t_{i}^{\alpha}+t_{j}^{\alpha}-|t_{i}-t_{j}|^{\alpha}\big)_{i,j=1}^{n}
\]
and the joint probability density function is given by:
\[
f_{\alpha,\beta}(\theta,\Sigma_{\alpha})=\frac{(2\pi)^{-\frac{n}{2}}}{\sqrt{\det\Sigma_{\alpha}}}\int_{0}^{\infty}\tau^{-\frac{n}{2}}e^{-\frac{1}{2\tau}\theta^{\top}\Sigma_{\alpha}^{-1}\theta}M_{\beta}(\tau)\,d\tau.
\]
\end{enumerate}
Here $E_{\beta}$ is the Mittag-Leffler (entire) function 
\[
E_{\beta}(z)=\sum_{n=0}^{\infty}\frac{z^{n}}{\Gamma(\beta n+1)},\qquad z\in\mathbb{C},
\]
and where $M_{\beta}$ is the so-called $M$-Wright probability density
function with Laplace transform 
\begin{equation}
\int_{0}^{\infty}e^{-s\tau}M_{\beta}(\tau)\,d\tau=E_{\beta}(-s).\label{eq:M_wright}
\end{equation}
The absolute moments of order $\delta>-1$ in $\mathbb{R}^{+}$ are
finite and turn out to be 
\begin{equation}
\int_{0}^{\infty}\tau^{\delta}M_{\beta}(\tau)\,d\tau=\frac{\Gamma(\delta+1)}{\Gamma(\beta\delta+1)}.\label{moments}
\end{equation}

The generalized grey Brownian motion has the following properties: 
\begin{enumerate}
\item For each $t\geq0$, the moments of any order are given by 
\[
\begin{cases}
\mathbb{E}(B_{\alpha,\beta}^{2n+1}(t)) & =0,\\
\noalign{\vskip4pt}\mathbb{E}(B_{\alpha,\beta}^{2n}(t)) & =\frac{(2n)!}{2^{n}\Gamma(\beta n+1)}t^{n\alpha}.
\end{cases}
\]
\item The covariance function has the form 
\begin{equation}
\mathbb{E}\big(B_{\alpha,\beta}(t)B_{\alpha,\beta}(s)\big)=\frac{1}{2\Gamma(\beta+1)}\big(t^{\alpha}+s^{\alpha}-|t-s|^{\alpha}\big),\quad t,s\geq0.\label{eq:auto-cv-gBm}
\end{equation}
\item For each $t,s\geq0$, the characteristic function of the increments
is 
\begin{equation}
\mathbb{E}\big(e^{i\theta(B_{\alpha,\beta}(t)-B_{\alpha,\beta}(s))}\big)=E_{\beta}\left(-\frac{\theta^{2}}{2}|t-s|^{\alpha}\right),\quad\theta\in\mathbb{R}.\label{eq:cf_gBm_increments}
\end{equation}
\end{enumerate}
All these properties may be summarized as follows. For any $0<\alpha<2$
and $0<\beta\leq1$, the ggBm $B_{\alpha,\beta}(t)$, $t\geq0$, is
$\frac{\alpha}{2}$-self-similar with stationary increments. We refer
to \cite{Mura_mainardi_09} for the proof and more details. This class
includes fBm for $\beta=1$, and Brownian motion (Bm) for $\alpha=\beta=1$.
We note also that Equation (\ref{cract func}) shows that ggBm, which
is not Gaussian in general, is a stochastic process defined only through
its first and second moments which is a property of Gaussian processes.

Finally, it was shown in \cite{Mura_Pagnini_08} that the ggBm $B_{\alpha,\beta}$
admits the following representation 
\begin{equation}
\big\{ B_{\alpha,\beta}(t),\;t\geq0\big\}\overset{d}{=}\big\{\sqrt{Y_{\beta}}B_{H}(t),\;t\geq0\big\},\label{gbm-rep}
\end{equation}
where $\overset{d}{=}$ denotes the equality of the finite dimensional
distribution and $B_{H}$ is a standard fBm with Hurst parameter $H=\alpha/2$.
$Y_{\beta}$ is an independent non-negative random variable with probability
density function $M_{\beta}(\tau)$, $\tau\geq0$.

\subsection{The p-variation of generalized grey Brownian motion}

This subsection is devoted to the study of the $p$-variation of ggBm.
The approach taken is inspired from the one used for the fBm. To do
this we will need to establish some properties satisfied by the ggBm
$B_{\alpha,\beta}$. 
\begin{lem}
The stationary sequence $\big(B_{\alpha,\beta}(j)-B_{\alpha,\beta}(j-1)\big)_{j\geq1}$
is mixing. Therefore it is ergodic.
\end{lem}
\begin{proof}
Since the process $B_{\alpha,\beta}(t)$, $t\geq0$, has stationary
increments then the sequence $\big(B_{\alpha,\beta}(j)-B_{\alpha,\beta}(j-1)\big)_{j\geq1}$
is stationary. To show that it is also mixing, it is sufficient to
prove the decay of correlations

\begin{equation}
\lim_{j\rightarrow+\infty}\mathrm{Cov}\Big(f\big(B_{\alpha,\beta}(1)\big),g\big(B_{\alpha,\beta}(j)-B_{\alpha,\beta}(j-1)\big)\Big)=0,\label{mix cond}
\end{equation}
for all bounded measurable $f,g$. It follows from the representation
(\ref{gbm-rep}) that
\begin{multline*}
\mathrm{Cov}\Big(f\big(B_{\alpha,\beta}(1)\big),g\big(B_{\alpha,\beta}(j)-B_{\alpha,\beta}(j-1)\big)\Big)\\
=\int_{0}^{\infty}M_{\beta}(\tau)\,\mathrm{Cov}\left(f\left(\tau^{\nicefrac{1}{2}}B_{H}(1)\right),g\left(\tau^{\nicefrac{1}{2}}\left(B_{H}(j)-B_{H}(j-1)\right)\right)\right)d\tau.
\end{multline*}
It is well known that the sequence $\left(B_{H}(j)-B_{H}(j-1)\right)_{j\geq1}$
is stationary, centered Gaussian with covariance function satisfying
\[
\mathrm{Cov}\big(B_{H}(1)-B_{H}(0),B_{H}(j)-B_{H}(j-1)\big)\xrightarrow[j\rightarrow+\infty]{}0.
\]
Therefore it is mixing, that is 
\[
\underset{j\rightarrow+\infty}{\lim}\mathrm{Cov}\left(f\left(B_{H}(1)\right),g\Big(\big(B_{H}(j)-B_{H}(j-1)\big)\Big)\right)=0,
\]
for all bounded measurable $f,g$. Now (\ref{mix cond}) follows from
the dominated convergence theorem.
\end{proof}
\begin{lem}
\label{lem: convergence}The sequence $\left(\frac{1}{n}\sum_{j=1}^{n}\big|B_{\alpha,\beta}(j)-B_{\alpha,\beta}(j-1)\big|^{p}\right)_{n\geq1}$
converges a.s. (and in $L^{1}$) to $\mathbb{E}\left(\left|B_{\alpha,\beta}(1)\right|^{p}\right)$.
\end{lem}
\begin{proof}
Since $\big(B_{\alpha,\beta}(j)-B_{\alpha,\beta}(j-1)\big)_{j\geq1}$
is ergodic then as a consequence of the ergodic theorem (Theorem 3.3
p.\ 413 of \cite{S96}), we obtain that $\Big(\frac{1}{n}\sum_{j=1}^{n}\big|B_{\alpha,\beta}(j)-B_{\alpha,\beta}(j-1)\big|^{p}\Big)_{n\geq1}$
converge a.s.\ (and in $L^{1}$) to $\mathbb{E}\left(\left|B_{\alpha,\beta}(1)\right|^{p}\right)$.
\end{proof}
As a consequence we obtain 
\begin{prop}
We have the following limit in probability 
\[
\lim_{n\rightarrow+\infty}n^{p\frac{\alpha}{2}-1}\sum_{j=1}^{n}\left|B_{\alpha,\beta}\left(\frac{j}{n}\right)-B_{\alpha,\beta}\left(\frac{j-1}{n}\right)\right|^{p}=\mathbb{E}\left(\left|B_{\alpha,\beta}(1)\right|^{p}\right).
\]
\end{prop}
\begin{proof}
It follows from the $\frac{\alpha}{2}$-self-similarity of $B_{\alpha,\beta}(t)$,
$t\geq0$, that for all $n\in\mathbb{N}$, we have the following equality
in law 
\[
Z_{n,p}:=n^{p\frac{\alpha}{2}-1}\sum_{j=1}^{n}\left|B_{\alpha,\beta}\left(\frac{j}{n}\right)-B_{\alpha,\beta}\left(\frac{j-1}{n}\right)\right|^{p}=\frac{1}{n}\sum_{j=1}^{n}\left|B_{\alpha,\beta}(j)-B_{\alpha,\beta}(j-1)\right|^{p}.
\]
This, together with the convergence in Lemma \ref{lem: convergence}
gives that $Z_{n,p}$ converges in law to $\mathbb{E}\left(\left|B_{\alpha,\beta}(1)\right|^{p}\right)$.
Since the limit is a constant so the convergence in probability follows. 
\end{proof}
As a consequence of ${\bf {(ii)}}$ and ${\bf {(iii)}}$ we get the
following result on the $p$-variation of the ggBm. 
\begin{prop}
\label{prop: variation}We have the following limit in probability
\[
V_{p,n}:=\sum_{j=1}^{n}\left|B_{\alpha,\beta}\left(\frac{j}{n}\right)-B_{\alpha,\beta}\left(\frac{j-1}{n}\right)\right|^{p}\xrightarrow[n\rightarrow+\infty]{}\left\{ \begin{array}{lll}
0 & \text{a.s. if} & p\alpha/2>1\\
\infty & \text{a.s. if} & p\alpha/2<1\\
\mathbb{E}\left(\left|B_{\alpha,\beta}(1)\right|^{p}\right) & \text{a.s. if} & p=2/\alpha.
\end{array}\right.
\]
\end{prop}
\begin{rem}
The ggBm is not a semimartingale. In addition, $B_{\alpha,\beta}$
cannot be of finite variation on $[0,1]$ and by scaling and stationarity
of the increment on any interval. 
\end{rem}
\begin{proof}
Indeed there is a subsequence such that $V_{p,n}$ converge almost
surely to $\infty$ for $p=1$ and $\alpha\in\left(0,2\right)$. If
$\alpha\in\left(1,2\right)$ we can choose $p\in\left(2/\alpha,2\right)$
such that $V_{p,n}$ converge to $0$ for some subsequence. This implies
that the quadratic variation of $B_{\alpha,\beta}$ is zero. If $\alpha\in\left(0,1\right)$
we can choose $p>2$ such that $2p/\alpha<1$ and the $p$-variation
of $B_{\alpha,\beta}$ must be infinite. So, in any case $B_{\alpha,\beta}$
can not be a semimartingale.
\end{proof}
\begin{rem}
\label{as conv} It follows from the representation (\ref{gbm-rep})
that the H{\"o}lder continuity of the trajectories of ggBm reduces
to the H{\"o}lder continuity of the fBm. Thus, with probability $1$,
we have 
\[
\sum_{j=1}^{2^{n}}\left|B_{\alpha,\beta}\left(\frac{j}{2^{n}}\right)-B_{\alpha,\beta}\left(\frac{j-1}{2^{n}}\right)\right|^{p}\xrightarrow[n\rightarrow+\infty]{}\left\{ \begin{array}{lll}
0 & \text{a.s. if} & p\alpha/2>1\\
\infty & \text{a.s. if} & p\alpha/2<1\\
\mathbb{E}\left(\left|B_{\alpha,\beta}(1)\right|^{p}\right) & \text{a.s. if} & p=2/\alpha.
\end{array}\right.
\]
\end{rem}

\subsection{Singularity of generalized grey Brownian motions with different parameters}

Since the law of the generalized grey Brownian motion is not Gaussian,
Feldman-H{\'a}jek dichotomy is no longer applicable. However, using
the same approach as for the fBm, we establish that the laws of two
ggBm processes are singular if the parameters are different. Let us
denotes by $P_{\alpha,\beta}$ (resp.\ $P_{\alpha',\beta'}$) the
probability measures generated by $B_{\alpha,\beta}$ (resp.\ $B_{\alpha',\beta'}$)
with $\alpha,\alpha'\in(0,2)$ and $\beta,\beta'\in(0,1]$.
\begin{thm}
\label{thm:ggnm-orthogonal}
\begin{enumerate}
\item For any $\beta,\beta'\in(0,1]$ such that $\Gamma(\nicefrac{\beta}{\alpha}+1)\neq\Gamma(\nicefrac{\beta'}{\alpha}+1)$,
the probability measures $P_{\alpha,\beta}$ and $P_{\alpha,\beta'}$
are singular with respect to each other for all $\alpha\in(0,2)$.
\item For $\alpha\neq\alpha'$ then for any $\beta,\beta'\in(0,1]$ the
probability measures $P_{\alpha,\beta}$ and $P_{\alpha',\beta'}$
are singular with respect to each other. 
\end{enumerate}
\end{thm}
\begin{rem}
Before proceeding to the proof of the theorem we will give some examples
of $\beta,\beta'\in(0,1]$ for which the condition $\Gamma(\nicefrac{\beta}{\alpha}+1)\neq\Gamma(\nicefrac{\beta'}{\alpha}+1)$
is verified. It is well known that there exists $\kappa\in(1,2)$
such that the $\Gamma$ function is strictly decreasing on $(0,\kappa]$
and strictly increasing on $[\kappa,\infty)$. So, for $\beta,\beta'\in(0,\alpha(\kappa-1)\wedge1]$
or $\beta,\beta'\in[\alpha(\kappa-1)\wedge1,1]$ the condition is
satisfied.
\end{rem}
\begin{proof}[Proof of Theorem\ \ref{thm:ggnm-orthogonal}]
It follows from Remark \ref{as conv} that 
\begin{equation}
\lim_{n\rightarrow+\infty}\sum_{j=1}^{2^{n}}\left|x\left(\dfrac{j}{2^{n}}\right)-x\left(\dfrac{j-1}{2^{n}}\right)\right|^{\nicefrac{2}{\alpha}}=\mathbb{E}\left(\left|B_{\alpha,\beta}(1)\right|^{2/\alpha}\right)=:\mu_{\alpha,\beta},\quad P_{\alpha,\beta}-a.s.,\label{eq:lim var 1}
\end{equation}
and 
\begin{equation}
\lim_{n\rightarrow+\infty}\sum_{j=1}^{2^{n}}\left|x\left(\dfrac{j}{2^{n}}\right)-x\left(\dfrac{j-1}{2^{n}}\right)\right|^{\nicefrac{2}{\alpha'}}=\mathbb{E}\left(\left|B_{\alpha',\beta'}(1)\right|^{2/\alpha'}\right)=:\mu_{\alpha',\beta'},\; P_{\alpha',\beta'}-a.s.\label{eq:lim var 2}
\end{equation}

In other words, $P_{\alpha,\beta}(C_{\mu_{\alpha,\beta}}^{\nicefrac{2}{\alpha}})=1$
and $P_{\alpha',\beta'}(C_{\mu_{\alpha',\beta'}}^{\nicefrac{2}{\alpha'}})=1$.
\begin{enumerate}
\item For $(\alpha,\beta)\neq(\alpha,\beta')$, using the representation
\eqref{gbm-rep}, the independence of $Y_{\beta}$ and $B_{H}$ and
the moments formula \eqref{moments} we obtain 
\[
\mathbb{E}\left(\left|B_{\alpha,\beta}(1)\right|^{\nicefrac{2}{\alpha}}\right)=\frac{\Gamma(\nicefrac{1}{\alpha}+1)}{\Gamma(\nicefrac{\beta}{\alpha}+1)}\mathbb{E}\left(\left|B_{H}(1)\right|^{\nicefrac{2}{\alpha}}\right),
\]
and 
\[
\mathbb{E}\left(\left|B_{\alpha,\beta'}(1)\right|^{\nicefrac{2}{\alpha}}\right)=\frac{\Gamma(\nicefrac{1}{\alpha}+1)}{\Gamma(\nicefrac{\beta'}{\alpha}+1)}\mathbb{E}\left(\left|B_{H}(1)\right|^{\nicefrac{2}{\alpha}}\right).
\]
Since $\Gamma(\nicefrac{\beta}{\alpha}+1)\neq\Gamma(\nicefrac{\beta'}{\alpha}+1)$
it is clear that for the limits \eqref{eq:lim var 1} and \eqref{eq:lim var 2}
are different, that is $\mu_{\alpha,\beta}\neq\mu_{\alpha,\beta'}$.
The singularity of $P_{\alpha,\beta}$ and $P_{\alpha,\beta'}$ follows
from the fact that the sets $C_{\mu_{\alpha,\beta}}^{\nicefrac{2}{\alpha}}$
and $C_{\mu_{\alpha,\beta'}}^{\nicefrac{2}{\alpha}}$ are disjoints.
\item For $(\alpha',\beta')\neq(\alpha,\beta)$ ($\alpha<\alpha'$ for example)
it follows from \textbf{(ii)} that $C_{\mu_{\alpha',\beta'}}^{\nicefrac{2}{\alpha'}}\subset C_{0}^{\nicefrac{2}{\alpha}}$.
Hence we deduce, from \textbf{(i)} that the measures $P_{\alpha,\beta}$
and $P_{\alpha',\beta'}$ are singular with respect to each other.\hfill $\qedhere$
\end{enumerate}
\end{proof}

\subsection*{Acknowledgments}

We would like to thank the financial support of the Laboratory LIBMA
form the University Cadi Ayyad Marrakech and the project I\&D: UID/MAT/04674/2013.

\bibliographystyle{alpha}

\end{document}